\documentclass[12pt, A4paper]{article}
\usepackage{amsmath,amssymb,amscd,amsthm}
\usepackage{textcomp}
\newcounter{numofprop}[section]

\newtheorem{theorem}{Theorem}[section]
\newtheorem{predl}[theorem]{Proposition}
\newtheorem{lemma}[theorem]{Lemma}

\newcommand{\C}{\mathbb C}
\newcommand{\N}{\mathbb N}
\newcommand{\Z}{\mathbb Z}
\renewcommand{\AA}{\mathcal A}
\renewcommand{\SS}{\mathcal S}
\newcommand{\DD}{\mathcal D}
\newcommand{\CC}{\mathcal C}
\newcommand{\ra}{\mathbin{\rightarrow}}

\renewcommand{\P}{\mathbb P}
\newcommand{\Proj}{\mathbf{Proj~}}
\newcommand{\Spec}{\mathbf{Spec~}}
\newcommand{\D}{\mathbf D}
\renewcommand{\O}{\mathcal O}

\renewcommand{\le}{\leqslant}
\renewcommand{\ge}{\geqslant}

\newcommand{\wdt}{\overline}

\DeclareMathOperator{\Hom}{\textup{Hom}}
\DeclareMathOperator{\HHom}{\underline{Hom}}
\DeclareMathOperator{\Ext}{\textup{Ext}}

\DeclareMathOperator{\QGr}{\mathrm{QGr}}
\DeclareMathOperator{\Gr}{\mathrm{Gr}}
\DeclareMathOperator{\qgr}{\mathrm{qgr}}
\DeclareMathOperator{\gr}{\mathrm{gr}}
\DeclareMathOperator{\Tors}{\textup{Tors}}
\DeclareMathOperator{\tors}{\textup{tors}}

\def\s{\sigma}
\def\x{\chi}
\def\X{\mathrm{X}}

\begin{document}

\author{A.\,Elagin}
\title{
On exceptional collections on some log Del Pezzo surfaces with one singular
point\footnote{
This work was partially supported by CRDF grant RUM1-2661-MO-05}}
\date{}
\maketitle

\section{Introduction.}

Let $X$ be an algebraic variety, let $\mathrm{coh}(X)$ be the category
of coherent sheaves on $X$.
One of ways to describe the derived category
$\D^b(\mathrm{coh}(X))$ is to construct an exceptional
collection in it.

An object $E$ of a triangulated $\C$-linear category $\CC$ is called 
\emph{exceptional}                   
if $\Hom(E,E)=\C$ and $\Hom^i(E,E)=0$ for $i\ne 0$.
An ordered collection $(E_1,\ldots, E_n)$ 
is called \emph{exceptional} if all $E_k$ are exceptional
and  $\Hom^i(E_k,E_l)=0$ for $k>l$ and all $i\in \Z$. A collection
in $\CC$ is said to be \emph{full} if objects of the collection
generate the category $\CC$.

Of course, not all categories possess a full exceptional collection,
but it is proved that bounded derived categories of sheaves on some special 
varieties  do. We mention two results of this kind.
Exceptional collections of sheaves on smooth Del Pezzo surfaces
were constructed and investigated by D.\,Orlov and S.\,Kuleshov~\cite{Or,KuO}.
Y.\,Kawamata proved that full exceptional collections exist on
toric varieties with quotient singularities.

In this paper we construct full exceptional collections of sheaves on
some special hypersurfaces in weighted projective spaces. Namely, 
we consider hypersurfaces of degree $4n-2$  in the weighted projective spaces
$\P(1,2,2n-1,4n-3)$.
These surfaces are singular non-toric Del Pezzo surfaces,
they have a unique singular point. 
We consider these singular surfaces as smooth stacks in the following sense.
For any commutative $\N$-graded algebra $S$ finitely generated over the field 
$\C=S_0$ we can define
a stack structure on the projective variety $\Proj S$ 
by the action of $\C^*$ on $\Spec S\setminus\{0\}$
associated with the grading.
A sheaf on such stack is by definition a $\C^*$-equivariant sheaf on 
$\Spec S\setminus\{0\}$.
According to~\cite{AKO}, the category 
$\mathrm{coh}^{\C^*}(\Spec S\setminus\{0\})$ of $\C^*$-equivariant 
coherent sheaves on $\Spec S\setminus\{0\}$ is equivalent to the 
quotient category $\qgr(S)=\gr(S)/\tors S$, where $\gr(S)$ is 
the category of finitely generated graded $S$-modules and $\tors S$ is
its subcategory of torsion modules. 
By Serre theorem, if $S$ is generated by its homogeneous component $S_1$, then
the category $\qgr(S)$ is equivalent to the category $\mathrm{coh}(\Proj S)$
of coherent sheaves on $\Proj S$. 

Since the surfaces we consider are embedded into the weighted 
projective space, they  are of the form
$\Proj A$, where $A=\C[x_0,\ldots,x_3]/(f)$ is the quotient of the algebra
of weighted polinomials.
In our case $A_1$ does not  generate
$A$ and in fact the categories 
$\mathrm{coh}^{\C^*}(\Spec A\setminus\{0\})\cong\qgr(A)$ and 
$\mathrm{coh}(\Proj A)$ are not equivalent. 
We construct exceptional collections in the first category,
all the work is done in terms of modules over the algebra $A$. 
                                                   
Any hypersurface
$X=\Proj S$ of degree~$d$ in the weighted projective space
$\P(a_0{,}\ldots,a_n)$ possesses the exceptional collection
$(S,S(1),\ldots,S(-\kappa{-}1))$ in the category $\D^b(\qgr (S))$, 
where  $S=\C[x_0{,}\ldots,x_n]/{(f(x))}$ is the quotient algebra and 
$\kappa=d-\sum a_i$.
In our case this gives the collection 
$(A,A(1),\ldots,A(2n))$, which is not full.
To construct a full collection we consider modules~$\x_j$ with
support in the singular point $P\in X$. They
correspond to the same skyscraper sheaf $\O_P$ 
but carry different equivariant structures.
Using modules $\x_j$ we can form the collection
$(A,\ldots,A(2n),\x_{2n},\x_{2n-1},\ldots,\x_4,\x_3)$. 
This collection is full but it is not exceptional.
To obtain an exceptional collection we make mutations with 
modules $\x_4$ and $\x_3$ in the above collection. 

Recall that left and right mutations $L_X(Y)$ and $R_Y(X)$ 
of objects $X$ and $Y$ in a triangulated category
are defined by the triangles
$$L_X(Y)\ra \Hom^{\cdot}(X,Y)\otimes X\ra Y\quad\text{and}\quad
X\ra \Hom^{\cdot}(X,Y)^*\otimes Y\ra R_Y(X).$$
A left mutation of an object over an exceptional collection
is defined by induction:
$L_{\langle E_1\rangle}(X)=L_{E_1}(X)$ and
$L_{\langle E_1,\ldots,E_r\rangle}(X)=%
L_{E_1}(L_{\langle E_2,\ldots,E_r\rangle}(X))$; a right mutation is defined 
similarly. For properties of mutations see~\cite{Bo}. 

One of full exceptional collections in  $\D^b(\qgr(A))$ 
is the following:
$$(A,A(1),\ldots,A(2n-2),G_{2n-1},A(2n-1),G_{2n},
A(2n),\x_{2n},\x_{2n-1},\ldots,\x_6,\x_5).$$
Here  
$G_{2n}{=}L_{\langle A(2n),\x_{2n},\x_{2n-2},\ldots,\x_8\,\x_6\rangle}(\x_4)[n-2]$
and 
$G_{2n-1}{=}L_{\langle A(2n-1),\x_{2n-1},\x_{2n-3},\ldots,\x_7\,\x_5\rangle}(\x_3)[n-2]$
are twisted ideals in the algebra $A$.
We prove this in section~\ref{sectionEC}, theorems~\ref{EC1}~and~\ref{full}.  

Given an exceptional collection $(E_1,\ldots,E_m)$ we can
define a graded associative algebra
$\AA=\oplus_{i \in\Z}(\oplus_{1\le k\le l\le m}\Hom^i(E_k,E_l))$, 
where multiplication is given by the composition law.
If this algebra is concentrated in degree $0$ 
(i.e., $\Hom^i(E_k,E_l)=0$ for all $k,l$ and $i\ne 0$), 
the triangulated category generated by the collection can be recovered from
the algebra $\AA$.
Explicitly, this category is equivalent to $\D^b(\mathrm{mod-}\AA)$ (see~\cite{Bo}). 
In general case the algebra $\AA$ also contains a lot of information
on the category generated by the collection.
We compute the algebra of morphisms for one of exceptional collections
in $\D^b(\qgr(A))$ in section~\ref{alg}.

\section{Definition and properties of basic objects.}

Fix an integer $n$, $n\ge 2$. 

Let $B$ be the graded algebra of polinomials
$\C[x_0,x_1,x_2,x_3]$ generated by homogeneous variables 
$x_0,x_1,x_2,x_3$ of degrees $1, 2, 2n-1, 4n-3$ respectively.

Then $\Proj B$ is the weighted projective space
$\P(1,2,2n{-}1,4n{-}3)$. It is a singular variety with three singular points,
these points are cyclic quotient singularities of types
$\frac 12(1,1,1)$, $\frac 1{2n-1}(1,2,-1)$, $\frac 1{4n-3}(1,2,2n-1)$.

\subsection{Definition and properties of surface $X$.}
\label{sectiongeom}

Consider a hypersurface $X$ in $\P(1,2,2n{-}1,4n{-}3)$ defined by
a general homogeneous polinomial of degree $d=4n-2$.
One can easily check that such surface is unique up to an automorphism of
the weighted projective space. To be more precise, the following lemma holds:

\begin{lemma}
Suppose $f(x)\in B$ is a homogeneous polinomial of degree $4n-2$ and
coefficients of monomials $x_0x_3, x_2^2$ and $x_1^{2n-1}$ in $f(x)$ 
are not equal to zero.
Then there exists a homogeneous change of coordinates in $\P(1,2,2n{-}1,4n{-}3)$
that takes $f(x)$ to
$ f(x')=x_0'x_3'+(x_1')^{2n-1}+(x_2')^2$.
\end{lemma}

Further on we assume that 
\begin{equation}
f(x)=x_0x_3+x_1^{2n-1}+x_2^2.
\end{equation}

Below we list some geometric properties of $X$.

\begin{predl}\label{onesing}
\
\begin{itemize}
\item[(a)] The surface $X$ is rational.
\item[(b)] The surface $X$ has unique singular point $P=(0{:}0{:}0{:}1)$.
\item[(c)] The point $P$ is a cyclic quotient singularity of type
$\frac1{4n-3}(1,4)$. The configuration of exceptional curves on the minimal
resolution of $P\in X$ looks as follows:

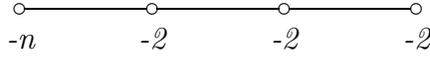
\begin{center}
\begin{picture}(250,40)
\put(52,25){\line(1,0){46}}
\put(102,25){\line(1,0){46}}
\put(152,25){\line(1,0){46}}
\put(50,25){\circle{4}}
\put(100,25){\circle{4}}
\put(150,25){\circle{4}}
\put(200,25){\circle{4}}
\put(45,10){-n}
\put(95,10){-2}
\put(145,10){-2}
\put(195,10){-2}
\end{picture}
\end{center}
\item[(d)] $X$ is a Del Pezzo surface with $\mathrm{Pic}~X=\Z$.
\end{itemize}
\end{predl}

\begin{proof}
(a,b,c) The proof can be done by direct calculation in coordinates.
The configuration of exceptonal curves on the minimal
resolution is determined by the expansion of $\frac{4n-3}4$ into 
a continious fraction (see~\cite[sect. 2.6]{Fu}): 
$$\frac{4n-3}4=n-\cfrac1{2-\cfrac1{2-\cfrac12}}\:.$$

(d) 
Consider the intersection of $X$ with the hyperplane $x_0=0$. 
Since this intersection is a curve in $\P(2,2n-1,4n-3)$ given by the 
irreducible 
polinomial $x_1^{2n-1}+x_2^2=0$, it is an irreducible divisor. 
The complement $X\cap \{x_0\ne 0\}$ is isomorphic to an affine plane
so that $\mathrm{Cl}(X\cap \{x_0\ne 0\})=0$.
We conclude that the group $\mathrm{Cl}{}~X$ is generated by
the hyperplane section of $X$ (see~\cite[prop. II.6.5]{Ha}).
Therefore $\mathrm{Pic}{}~X=\Z$.

From the adjunction formula for $\P^3(1,2,2n-1,4n-3)$ and $X$ it
follows that the anticanonical divisor $-K_X$ is effective.
Since $\mathrm{Cl}{}~X=\Z$, the divisor $-K_X$ is ample. 
\end{proof}

Below we give an explicit construction of $X$ in terms of birational
transformations.

Suppose $\mathbf{F}_n$  is the Hirzebruch surface, 
$\mathbf{F}_n=\P_{\P^1}(\O\oplus\O(n))$. Blowing up
three points as shown on the picture below, we obtain the surface $\~X$.

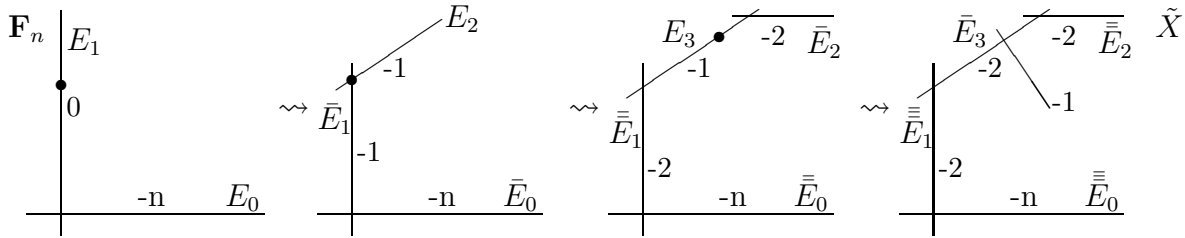
\begin{center}
\begin{picture}(430,100)
\put(-7,80){$\mathbf{F}_n$}

\put(0,13){\line(1,0){90}}
\put(13,5){\line(0,1){85}}
\put(13,62){\circle*{4}}

\put(75,16){$E_0$}
\put(15,75){$E_1$}
\put(42,16){-n}
\put(15,50){\small 0}

\put(95,50){$\rightsquigarrow$}

\put(110,13){\line(1,0){85}}
\put(123,5){\line(0,1){65}}
\put(117,60){\line(3,2){40}}
\put(123,64){\circle*{4}}

\put(180,16){$\bar E_0$}
\put(110,45){$\bar E_1$}
\put(158,84){$E_2$}
\put(125,33){\small -1}
\put(152,16){-n}
\put(135,65){\small -1}

\put(205,50){$\rightsquigarrow$}

\put(220,13){\line(1,0){85}}
\put(233,5){\line(0,1){65}}
\put(227,57){\line(3,2){50}}
\put(267,88){\line(1,0){38}}
\put(262,80){\circle*{4}}

\put(290,16){$\bar{\bar E}_0$}
\put(220,40){$\bar{\bar E}_1$}
\put(240,78){$E_3$}
\put(295,74){$\bar E_2$}
\put(235,27){\small -2}
\put(262,16){-n}
\put(250,65){\small -1}
\put(278,77){\small -2}

\put(315,50){$\rightsquigarrow$}

\put(330,13){\line(1,0){85}}
\put(343,5){\line(0,1){65}}
\put(337,57){\line(3,2){50}}
\put(377,88){\line(1,0){38}}
\put(367,83){\line(2,-3){20}}

\put(400,16){$\bar{\bar{\bar E}}_0$}
\put(330,40){$\bar{\bar{\bar E}}_1$}
\put(350,78){$\bar{E}_3$}
\put(405,74){$\bar{\bar E}_2$}
\put(388,52){\small -1}
\put(345,27){\small -2}
\put(372,16){-n}
\put(388,77){\small -2}
\put(360,65){\small -2}

\put(427,80){$\~X$}
\end{picture}
\end{center}

\begin{predl}
The surface $\~X$ is a minimal
resolution of singularities of $X$; the exceptional divisor consists
of the proper transforms of $E_0, E_1, E_2, E_3$. 
\end{predl}

\begin{proof}
The proof follows from the classifiction of Del Pezzo surfaces over $\C$ 
of Picard rank one with  unique cyclic  quotient singularity, made 
in~\cite[theorem 4.1]{Ko}. 
We skip the details. 
\end{proof}

\subsection{Definition and properties of basic modules.}
\label{defA}
Let $A$ be the graded algebra 
$$A=B/(f)=\C[x_0,x_1,x_2,x_3]/{(x_0x_3+x_1^{2n-1}+x_2^2)},$$ 
it is the homogeneous coordinate algebra of $X$.
For any graded $A$-module $M$ let $M(k)$ be the same module
with grading shifted by $k$: $M(k)_i=M_{k+i}$.
Now we define $A$-modules that will be useful for a construction of
exceptional collections.

Consider $A$-modules 
$$\x=\x_0=A/(x_0,x_1,x_2)\cong\C[x_3]\quad\text{and}\quad\x_j=\x(j).$$
Obviously, the support of $\x_j$ is the singular point $P\in X$.
The homogeneous components $(\x_j)_k$ of $\x_j$ are one-dimensional for 
$k\ge -j$, $k \equiv -j\pmod{4n-3}$ and zero in other cases.

There exists a monomorphism $\x_k\ra \x_{k+4n-3}$, it is 
given by multiplication by $x_3$. 
Note that the cokernel of this monomorphism is a torsion module,
hence the modules $\x_k$ and $\x_{k+4n-3}$ correspond to isomorphic
objects in $\qgr(A)$. We see that in $\qgr(A)$ there exist
$4n-3$ nonisomorphic objects $\x_j$.

Define $A$-modules $Q_{j+2r,j}$ by
$$Q_{j+2r,j}=\left(A/{(x_0,x_1^{r+1},x_2)}\right)(j+2r)\cong%
\left(\C[x_1,x_3]/{(x_1^{r+1})}\right)(j+2r)$$
for any integers $j$ and $r$ such that $0\le r < 2n-1$.
As above, their support is the singular point $P\in X$.
Homogeneous components of $Q_{j+2r,j}$ of degree $k$ are one-dimensional
for $k\ge -(j+2r)$,
\mbox{$k\equiv -(j+2r),\ldots, -(j+2),-j\pmod{4n-3}$} and zero in
other cases. It follows from the definitions that $\x_j$ and
$Q_{j,j}$ are isomorphic.

\smallskip
Consider the sequences ($r\ge 0, s>0, r+s<2n-1$)
\begin{equation}
0\ra Q_{j+2r,j}\ra Q_{j+2r+2s,j}\ra Q_{j+2r+2s,j+2r+2}\ra 0,
\label{qqq}
\end{equation}
where the first map is multiplication by
$x_1^s$ and the second one is a factorization. It can be checked
that these sequences are exact.
As an important special case of~(\ref{qqq}), we obtain the following
exact sequences:
\begin{gather}
0\ra Q_{j+2r-2,j}\ra Q_{j+2r,j}\ra \x_{j+2r}\ra 0\label{sur}\quad\text{and}\\
0\ra \x_j\ra Q_{j+2r,j}\ra Q_{j+2r,j+2}\ra 0.\label{inj}
\end{gather}
From~(\ref{sur}) it follows that the module $Q_{j+2k,j}$ has a filtration
\begin{equation}
 0\subset Q_{j,j}\subset Q_{j+2,j}\subset Q_{j+4,j}\subset\ldots\subset  
Q_{j+2k-2,j}\subset Q_{j+2k,j}
\label{filtr}
\end{equation}
with quotients isomorphic to $\x_j, \x_{j+2}, \x_{j+4}, \ldots ,\x_{j+2k-2},
\x_{j+2k}$.

Consider the sequence
\begin{equation}
0\ra \x_j\xrightarrow{x_2} \left(A/{(x_0,x_1)}\right)(j{+}2n{-}1)\ra%
\x_{j+2n-1}\ra 0,
\label{xqx}
\end{equation}
where the first map is multiplication by $x_2$
and the second one is a factorization.
Since $A/{(x_0,x_1)}\cong \C[x_2,x_3]/{(x_2^2)}$, this sequence is exact.

\begin{lemma}\label{kozhul}
There exist exact sequences
\begin{gather*}
0\ra A(-3)\ra A(-1)\oplus A(-2)\ra A\ra A/{(x_0,x_1)}\ra 0\quad\text{and}\\
0\ra A(-2n)\ra A(-1)\oplus A(-2n+1)\ra A\ra Q_{0,-4n+4}\ra 0.
\end{gather*}
\end{lemma}

\begin{proof}                 
We can take the Kozhul complexes for $(x_0,x_1)$ and $(x_0,x_2)$ as required 
sequences.
Indeed, the pairs $(x_0,x_1)$ and $(x_0,x_2)$ are regular;
therefore the Kozhul complexes are resolutions of $A/{(x_0,x_1)}$
and $A/{(x_0,x_2)}\cong A/{(x_0,x_1^{2n-1},x_2)}\cong Q_{0,-4n+4}$.
The latter isomorphisms follow from the definitions of $A$ and $Q_{0,-4n+4}$.  
\end{proof}
                                                             
\subsection{$\Ext$ groups for basic modules.}
\label{sectionExt}
In this section we calculate $\Ext$ groups for objects
$A(i)$ and $\x_j$ in the category $\qgr(A)$. 

Recall that $\qgr(A)$ is the quotient category $\gr(A)/\tors A$,
where $\gr(A)$ is the category of finitely generated graded $A$-modules and
$\tors A\subset\gr(A)$ is its dense subcategory of torsion modules.
Similarly, let $\Gr(A)$ be the category of all graded $A$-modules,
$\Tors A\subset\Gr(A)$ its subcategory of torsion modules,
and $\QGr(A)$ the quotient category $\Gr(A)/\Tors A$.
Since the category $\QGr(A)$ has enough injective objects, one can define
the functors $\Ext^i(M,\cdot)$ on $\QGr(A)$ as derived functors of $\Hom(M,\cdot)$.
This also defines functors $\Ext^i$ on $\qgr(A)$, see~\cite{AZ}
for details.

The following well-known (see, e.g.,~\cite{Ba}) result about 
$\Ext$ groups on weighted projective spaces is needed for the sequel.

\begin{predl}\label{forpolinoms}
Suppose $B=\C[x_0,x_1,\ldots,x_m]$ is the graded algebra of polinomials 
in variables $x_i$ of degrees $a_i$ respectively;
then the following equalities hold in $\qgr(B)$ :
\begin{enumerate}
\item[] $\Hom (B,B(k))=B_k$ for all $k\ge 0$, $\Hom (B,B(k))=0$ otherwise;
\item[] $\Ext^i (B,B(k))=0$ for all $k$ and $i$ such that $1\le i \le m-1$;
\item[]	$\Ext^m (B,B(k))=B_{-s-k}^*$ for all $k\le -s$, where $s=\sum a_i$;
$\Ext^m (B,B(k))=0$ in other cases.
\end{enumerate}
Moreover, there exist natural isomorphisms for all $i=0,\ldots,m$:

\medskip
$\Ext^i(B,N)\cong \Ext^{m-i}(N,B(-s))^*,\quad\text{where } s=\sum a_i$.
\end{predl}

The next proposition is the main result of the subsection.
\begin{predl}\label{Ext}
Let $A$ be the algebra defined in subsection~\ref{defA}.
Then the following equalities hold in $\qgr(A)$ :                 
\begin{enumerate}
\item[(a)] $\Hom(A(k),A(l))=A_{l-k}$ for all $k\le l$, 

$\Ext^2(A(k),A(l))=A_{k-l-(2n+1)}^*$ if $k-l\ge 2n+1$, 

$\Ext^i(A(k),A(l))=0$ otherwise.

\item[(b)] $\Hom(A(k),\x_j)=\C$ if $k\equiv j \pmod{4n-3}$, 

$\Ext^i(A(k),\x_j)=0$ in other cases.

\item[(c)] $\Ext^2(\x_k, A(j))=\C$ if $j \equiv k-(2n+1) \pmod{4n-3}$,

$\Ext^i(\x_k, A(j))=0$ in other cases.

\item[(d)] $\Hom(\x_k,\x_j)=\C$ if $k \equiv j \pmod{4n-3}$,

$\Ext^1(\x_k,\x_j)=\C$ if $j \equiv k-2 \pmod{4n-3}\quad\text{or}\quad 
j\equiv k-(2n-1) \pmod{4n-3}$,

$\Ext^2(\x_k,\x_j)=\C$ if $j \equiv k-(2n+1) \pmod{4n-3}$,

$\Ext^i(\x_k,\x_j)=0$ otherwise.
\end{enumerate}
\end{predl}

\begin{proof}
Obviously, $\Ext^i(M,N)\cong \Ext^i(M(j),N(j))$, therefore 
it suffices to consider the case $k=0$.
First we prove the following lemma.
\begin{lemma}\label{basechange}
Suppose $B=\C[x_0,\ldots,x_m]$ is the polinomial algebra
with grading $\deg(x_i)=a_i$,
$f\in B$ is a homogeneous polinomial of degree $d$,
and $A=B/(f)$ is the quotient algebra. Then for any $A$-module $M$
there exists a natural isomorphism
$\Ext_{\QGr(A)}^i(A,M)\cong \Ext_{\QGr(B)}^i(B,M)$.
\end{lemma}

\begin{proof}  

Define the  functors
$\s'^*\colon\Gr(B){\to}\Gr(A)$ by $\s'^*(N)=N\otimes_BA$ and 
$\s'_*\colon\Gr(A)\to \Gr(B)$ by $\s'_*(M)=M$. Then $\s'^*$ is the left
adjoint functor to $\s'_*$. Functors  $\s'^*$ and $\s'_*$ can be extended 
to adjoint functors $\s^*$ and $\s_*$, defined on 
categories $\QGr(B)$ and $\QGr(A)$.
Since $\QGr(B)$ has enough $\s^*$-acyclic objects
(for any $M\in\QGr(B)$ there exists an epimorphism
$\oplus_{\alpha} B(k_{\alpha})\ra M$), 
the derived functor $L\s^*$ is well-defined; since $\s_*$ is exact,
the derived functor $R\s_*=L\s_*$ is well-defined.
According to~\cite[lemma 15.6]{Ke}, the functors $L\s^*$ and $R\s_*$  
are adjoint, i.e.,
for any objects $X\in \D^+(\QGr(A))$ and $Y\in \D^-(\QGr(B))$ 
we have
\begin{equation}
\Hom_{\D(\QGr(A))}(L\s^*(Y),X)\cong \Hom_{\D(\QGr(B))}(Y,R\s_*(X)).
\label{adj3}
\end{equation}

The required isomorphism is obtained by applying the above
formula to the (-j)-complex $X=M[j]$ and the  0-complex $Y=B[0]$.
\end{proof}
     
Now we can prove proposition~\ref{Ext}.

(a)
From the above lemma it follows that 
$\Ext^i_{\QGr(A)}(A,A(l))\cong \Ext^i_{\QGr(B)}(B,A(l))$.
We compute groups $\Ext^i_{\QGr(B)}(B,A(l))$ by applying
the functor $\Hom(B,\cdot)$ to the exact sequence 
$$0\ra B(l{-}(4n{-}2))\xrightarrow{\cdot f(x)} B(l)\ra A(l)\ra 0.$$
Groups $\Ext^i(B,B(j))$ are calculated in proposition~\ref{forpolinoms};
we obtain the result by straightforward calculations.

(b,d) 
Let $D_{x_3}=\{x_3{\ne}0\}$ be the neighbourhood of the singular point $P\in X$,
let $s\colon D_{x_3}\ra X$ be its embedding into $X$; $s$
corresponds to the localization morphism $A\ra A[x_3^{-1}]$.

Consider the following functors associated with the embedding $s$.
The direct image functor $s_*\colon \Gr(A[x_3^{-1}])\ra \QGr(A)$ is defined 
by $s_*(M)=M$ 
and the inverse image functor $s^*\colon \QGr(A)\ra \Gr(A[x_3^{-1}])$ is the 
localization. Note that the localization of a torsion module
is zero; this shows that $s^*$ is well-defined.
Since functors $s_*$ and  $s^*$ are exact and $s_*$ is the
 right adjoint functor to $s^*$, we have natural isomorphisms
\begin{equation}
R^i\Hom_{\QGr(A)}(M,s_*(N))\cong R^i\Hom_{A[x_3^{-1}]}(s^*(M),N) \label{adj}
\end{equation}
for any $A$-module $M$ and any $A[x_3^{-1}]$-module $N$.

Now consider the module $\X=A[x_3^{-1}]/{(x_1,x_2)}\cong\C[x_3,x_3^{-1}]$ 
over the algebra $A[x_3^{-1}]$.
Note that the submodule of $s_*(\X)$ generated by its
components of nonnegative degree, is isomorphic to the module $\x$. 
Since the quotient module $s_*(\X)/\x$ is a torsion module,
we see that $\x$ and $s_*(\X)$ are isomorphic in the category $\QGr(A)$. 
For the same reason, $\x(j)\cong s_*(\X(j))$.

(b) From formula~(\ref{adj}) for $M=A$ and $N=\X(j)$ it follows that
$\Ext^i_{\QGr(A)}(A,\x_j)\cong \Ext^i_{\Gr(A[x_3^{-1}])}(A[x_3^{-1}],\X(j))$.
But the functor $\Hom(A[x_3^{-1}],\cdot)$ is exact on the category of 
graded \mbox{$A[x_3^{-1}]$-modules}, in fact, we have
$\Hom(A[x_3^{-1}],N)=N_0$. This makes further calculation trivial.

(d) Now we compute the groups $\Ext(\x,\x_j)$.
Since $s^*s_*\X\cong\X$, applying formula~(\ref{adj}) to 
$M=s_*(\X)$, $N=\X(j)$ we obtain  
$\Ext^i_{\QGr(A)}(\x,\x_j)\cong \Ext^i_{\Gr(A[x_3^{-1}])}(\X,\X(j))$.

Note that the algebra $A[x_3^{-1}]$ is isomorphic to 
$\C[x_1,x_2,x_3,x_3^{-1}]$ and the following Kozhul  complex
is a free resolution of the module
$\X=A[x_3^{-1}]/(x_1,x_2)$ 
$$ 
0\ra A[x_3^{-1}](-2n-1)\ra A[x_3^{-1}](-2n+1)\oplus A[x_3^{-1}](-2)\ra %
A[x_3^{-1}]\ra \X\ra 0. 
\label{resolutionX}
$$
According to the definition of $\Ext(\X,\X(j))$, we apply the functor 
$\Hom(\cdot,\X(j))$ to the above resolution and take homologies.
Using the results of part (b), we get:
$$
\Hom(\X,\X)=\C,\;
\Ext^1(\X,\X(-2))=\C,\;
\Ext^1(\X,\X(-2n+1))=\C,  \;
\Ext^2(\X,\X(-2n-1))=\C,
$$
other $\Ext$ groups are zero. Part (d) of the proposition is proved.

(c) 
The required equalities follow from part (b) and the following lemma.

\begin{lemma}\label{duality}
Suppose $B=\C[x_0,\ldots,x_m]$ is the polinomial algebra
with grading $\deg(x_i)=a_i$, $f\in B$ is a homogeneous polinomial
of degree $d$, and $A=B/(f)$ is the quotient algebra. Then
for any $A$-module $M$ there exists a natural isomorphism
$\Ext^i_{\QGr(A)}(A,M)\cong \Ext^{m-1-i}_{\QGr(A)}(M,A(d-s))^{*}$, 
where $s=\sum a_i$.
\end{lemma}

\begin{proof}
The proof is by reduction to duality for weighted projective space.
We have
\begin{equation}
\Ext^i_{\QGr(A)}(A,N)\cong\Ext^i_{\QGr(B)}(B,N)\cong%
\Ext^{m-i}_{\QGr(B)}(N,B(-s))^*.
\label{chain}
\end{equation}
The first isomorphism above follows from lemma~\ref{duality} and  
the second one follows from proposition~\ref{forpolinoms}.
                             
Define the functor $\s'^!\colon\Gr(B)\to\Gr(A)$ by 
$\s'^!(N)=\oplus_{k\in \Z} \Hom_{\Gr(B)}(A,N(k))[-k]=\HHom_B(A,N)$. 
Then $\s'^!$ is the right adjoint to $\s'_*$.
The functors $\s'_*$ and $\s'^!$ can be extended to adjoint functors
$\s_*$ and $\s^!$, defined on the categories $\QGr(A)$ and $\QGr(B)$.
Since $\QGr(B)$ has enough injective objects,
the derived functor $R\s^!: \D^+(\QGr(B))\to\D^+(\QGr(A))$ is well-defined.
From~\cite[lemma 15.6]{Ke} it follows that the derived functors
$L\s_*$ and $R\s^!$ are adjoint, i.e., for any
objects $X\in \D^-(\QGr(A))$ and $Y\in \D^+(\QGr(B))$ we have
\begin{equation}
\Hom_{\D(\QGr(B))}(L\s_*(X),Y)\cong \Hom_{\D(\QGr(A))}(X,R\s^!(Y))
\label{adj4}
\end{equation}
If we put $X=N[0]$ and $Y=B(-s)[j]$ in formula~(\ref{adj4}), we obtain
\begin{equation}
\Ext^j_{\QGr(B)}(N,B(-s))\cong\Hom_{\D(\QGr(A))}(N[0],R\s^!(B(-s)[j])).
\label{chain2}
\end{equation}
We need to compute $R^i\s^!(B(-s)[j])=R^{i+j}(\s^!B(-s))=R^{i+j}\HHom(A,B(-s))$.

First we calculate the groups $R^i\HHom(B(r),B(-s))$ for any $r$. 
From definitions it follows that
$R^i\HHom(B(r),B(-s))=\oplus_{k\in \Z} \Ext^i_{\QGr(B)}(B(r),B(-s+k))[-k]$.
Suppose $i>0$. Then $\Ext^i(B(r),B(-s+k))=0$ for $k\gg 0$ by 
proposition~\ref{forpolinoms}. Hence the object
$R^i\HHom(B(r),B(-s))$ equals zero in the quotient category $\QGr(B)$.
For $i=0$ we have $\HHom(B(r),B(-s))=B(-s-r)$ by proposition~\ref{forpolinoms}.

Now consider the exact sequence 
$0 \ra B(-d) \xrightarrow{\cdot f} B \ra A \ra 0$.
Applying the functor $\HHom(\cdot,B(-s))$ to this sequence, we obtain
a long exact sequence of derived functors. 
We already know terms $R^i\HHom(B,B(-s))$ and $R^i\HHom(B(-d),B(-s))$; 
finally we get
$$
\HHom(A,B(-s))=0,\quad
R^1\HHom(A,B(-s))=A(d-s),\quad
R^i\HHom(A,B(-s))=0\quad\text{for } i>1.
$$
Therefore, the object $R\s^!(B(-s))\in\D^+(\QGr(A))$ is 
the $1$-complex $A(d-s)[-1]$.

Combining~(\ref{chain}) and~(\ref{chain2}), we have
$$\Ext^i_{\QGr(A)}(A,N)\cong\Hom_{\D(\QGr(A))}(N[0],A(d-s)[m-i-1])^*%
\cong\Ext^{m-1-i}_{\QGr(A)}(N,A(d-s))^*.$$
This completes the proof of the lemma.
\end{proof}

Proposition~\ref{Ext} is completely proved. 
\end{proof}

\begin{lemma}\label{compos}
The following maps, given by Yoneda multiplication: 
\begin{gather*}
\Hom(A(j),\x_j)\otimes\Ext^2(\x_{j+2n+1},A(j))\ra \Ext^2(\x_{j+2n+1},\x_j)
\quad\text{and}\\
\Ext^2(\x_{j+2n+1},A(j))\otimes\Hom(A(j+2n+1),\x_{j+2n+1})\ra 
\Ext^2(A(j+2n+1),A(j))
\end{gather*}
are nontrivial.
\end{lemma}

\begin{proof} 
Let us prove that the first map is nontrivial.
For the second map the proof is analogous.

Consider exact sequences from subsection~\ref{defA} twisted by $j$: 
\begin{gather*}
0\ra A(j-3)\ra A(j-1)\oplus A(j-2)\ra A(j)\ra%
\left(A/{(x_0,x_1)}\right)(j)\ra 0\quad\text{and} \\
0\ra \x_{j+2n-2}\ra \left(A/{(x_0,x_1)}\right)(j)\ra \x_j\ra 0.
\end{gather*}
We claim that the quotient epimorphisms 
$A(j)\ra \left(A/{(x_0,x_1)}\right)(j)$ and
$\left(A/{(x_0,x_1)}\right)(j)\ra \x_j$
induce isomorphisms between groups
$\Ext^2(\x_{j+2n+1},M)$, where
$M=A(j)$, $\left(A/{(x_0,x_1)}\right)(j)$ and $\x_j$.
To check this, it suffices to apply the functor $\Hom(\x_{j+2n+1},\cdot)$
to the above sequences and to use proposition~\ref{Ext}.
\end{proof}

\section{Construction of exceptional collections.}
\label{sectionEC}

In this section we construct full exceptional collections
in the bounded derived category of $\qgr(A)$, where $A$ is the algebra,
defined in subsection~\ref{defA}.

Let $G_{j}=(x_0,x_1^{n-1},x_2)(j)$ and $H_{j}=(x_0,x_1,x_2)(j)$
be the twisted ideals in~$A$.
           
\begin{theorem}\label{EC1}
The collection of objects
$$
(A, A(1), A(2),\ldots, A(2n-1), 
G_{2n}, A(2n), H_{2n+1}, \x_{2n}, \x_{2n-1},
\ldots,\x_6, \x_5)
$$
in $\qgr(A)$ is exceptional for $n>2$.
\end{theorem}

To prove this theorem, we need two lemmas.

Recall that $Q_{j+2k,j}$ denotes the 
quotient module $\left(A/{(x_0,x_1^{k+1},x_2)}\right)(j)$.

\begin{lemma}\label{aboutQ}
For any $0\le k\le n-2$ the following statements hold:
\begin{itemize}
\item[(a)]
$Q_{j+2k,j}[-k]=L_{\langle\x_{j+2k},\x_{j+2k-2},\ldots,\x_{j+2}\rangle}(\x_j)$.

\item[(b)]
$\Ext^i(Q_{j+2k,j},\x_{j+2l})=0$ for all $i$ and $k,l=0,\ldots,n-2$ 
except for two cases:\\
$\Hom(Q_{j+2k,j},\x_{j+2k})=\C$, this space is generated by the epimorphism
\mbox{$Q_{j+2k,j}\ra \x_{j+2k}$};\\
$\Ext^2(Q_{j+2k,j},\x_{j+2n-4})=\C$ if $0\le k\le n-2$.

\item[(c)]
$\Ext^i(\x_{j+2k+2},Q_{j+2k,j})=0$ if $i\ne 1$;
 $\Ext^1(\x_{j+2k+2},Q_{j+2k,j})=\C$. 
A nonzero element of \: $\Ext^1(\x_{j+2k+2},Q_{j+2k,j})$ 
is represented by an extension of type~(\ref{sur}); 
the composition
$\Hom(Q_{j+2k,j}, \x_{j+2k})\otimes\Ext^1(\x_{j+2k+2},Q_{j+2k,j})
\ra \Ext^1(\x_{j+2k+2},\x_{j+2k})$ is nontrivial.

\item[(d)]
For $k<n{-}2$ the collection of objects 
$(\x_{j+2k}, \x_{j+2k-2},\ldots,\x_{j+2},\x_j)$
in $\qgr(A)$ is exceptional. The object $Q_{j+2k,j}=%
L_{\langle\x_{j+2k},\x_{j+2k-2},\ldots,\x_{j+2}\rangle}(\x_j)[k]$ 
is exceptional.
\end{itemize} 
\end{lemma}

\begin{proof}
The proof is by induction on $k$. Case $k=0$ is trivial. 
Assume that the lemma is proved for $k-1$.

(a)
By the induction hypothesis we have  
$Q_{j+2k-2,j}[1-k]=L_{\langle\x_{j+2k-2}, \ldots, \x_{j+2}\rangle}(\x_j)$.
We need to check that $L_{\x_{j+2k}}(Q_{j+2k-2,j})=Q_{j+2k,j}[-1]$. 
According to part (c) of the inductive assumption, 
the mutation $L_{\x_{j+2k}}(Q_{j+2k-2,j})$ is given by the triangle
$ Q_{j+2k,j}[-1]\ra \x_{j+2k}[-1]\ra Q_{j+2k-2,j}$,
which is defined by the exact sequence
\begin{equation}
0\ra Q_{j+2k-2,j}\ra Q_{j+2k,j}\ra \x_{j+2k}\ra 0. \label{s1}
\end{equation}

\smallskip
(b) To prove this, we apply the functor $\Hom(\cdot,\x_{j+2l})$ 
to the exact sequence~(\ref{s1}), taking into account
proposition~\ref{Ext} and part (b) of the inductive assumption.
In the case $l=k-1$, we use part (c) of the  inductive assumption
to check that the arising morphism is nonzero.

\smallskip
(c)
Since the object $Q_{j+2k+2,j}$ has a filtration~(\ref{filtr}), 
it follows from 
proposition~\ref{Ext} that $\Ext^i(\x_{j+2k+2},Q_{j+2k-2,j})=0$.
Now it remains to apply 
the functor $\Hom(\x_{j+2k+2},\cdot)$ to the exact sequence~(\ref{s1}).

\smallskip
Part (d) follows from proposition~\ref{Ext} and general properties of
mutations in exceptional collections, see~\cite{Bo}.
Lemma~\ref{aboutQ} is proved. 
\end{proof}

By definitions of the modules $G_{2n}, Q_{2n,4}$, and  $H_{2n+1}$, 
we have the following exact sequences:
\begin{gather}
0\ra G_{2n}\ra A(2n)\ra Q_{2n,4}\ra 0\quad\text{and}  \label{s2}\\
0\to H_{2n+1}\to A(2n+1)\to \x_{2n+1}\to 0. \label{s4}
\end{gather}

It can be easily checked (see lemma~\ref{aboutQ}a) 
that sequences~(\ref{s2}) and~(\ref{s4})
determine mutations 
$H_{2n+1}=L_{A(2n+1)}(\x_{2n+1})$ and
$G_{2n}=L_{A(2n)}(Q_{2n,4})=%
L_{\langle A(2n),\x_{2n},\x_{2n-2}, \ldots, \x_8,\x_6\rangle}(\x_4)[n-2]$. 
\begin{lemma}\label{Gexceptional}
Objects $G_{2n}$ and (for $n>2$) $H_{2n+1}$ are exceptional 
in the category $\qgr(A)$.
\end{lemma}

\begin{proof}
Since $H_{2n+1}$ is the  mutation $L_{A(2n+1)}(\x_{2n+1})$
in the exceptional (for $n>2$) collection $(A(2n+1),\x_{2n+1})$,
it is exceptional. Thus, we must show that
$\Hom(G_{2n},G_{2n})=\C$ and $\Ext^i(G_{2n},G_{2n})=0$ for $i\ne 0$. 
We do this in $6$ steps.

1. $\Ext^i(A(2n),G_{2n})=0$.

This follows from $G_{2n}=L_{A(2n)}(Q_{2n,4})$.

2. $\Ext^i(Q_{2n,4},Q_{2n-2,4})=0$.

To prove this, consider the long exact sequence obtained from
the sequence of the type~(\ref{sur})
\begin{equation}
0\ra Q_{2n-2,4}\ra Q_{2n,4}\ra \x_{2n}\ra 0  \label{s3}
\end{equation} 
by applying the functor $\Hom(\cdot,Q_{2n-2,4})$. 
All required $\Ext$ groups are computed in lemma~\ref{aboutQ}c,d.

3. $\Hom(Q_{2n,4},Q_{2n,4})=\C$, $\Ext^2(Q_{2n,4},Q_{2n,4})=\C$,
$\Ext^i(Q_{2n,4},Q_{2n,4})=0$ for $i\ne 0,2$.

This can be proved by applying the functor $\Hom(Q_{2n,4},\cdot)$
to the exact sequence~(\ref{s3}). Take into account lemma~\ref{aboutQ}b
and results of step $2$.

4. $\Ext^2(Q_{2n,4},A(2n))=\C$, $\Ext^i(Q_{2n,4},A(2n))=0$ for $i\ne 2$,
\newline
moreover the map $\Ext^2(Q_{2n,4},A(2n))\otimes \Hom(\x_4,Q_{2n,4})\ra
\Ext^2(\x_4,A(2n))$ is nontrivial.

Note that $Q_{2n,6}\in \langle\x_6, \x_8, \ldots, \x_{2n}\rangle$
and $A(2n)\in\langle\x_6, \x_8, \ldots, \x_{2n}\rangle^{\perp}$.
It follows that  $\Ext^i(Q_{2n,6},A(2n))=0$ for all $i$.
Now apply the functor $\Hom(\cdot,A(2n))$ to the exact sequence
$0\ra \x_4\ra Q_{2n,4}\ra Q_{2n,6}\ra 0$.

5. $\Ext^1(Q_{2n,4},G_{2n})=\C$, $\Ext^i(Q_{2n,4},G_{2n})=0$ for $i\ne 1$.  

Apply the functor $\Hom(Q_{2n,4},\cdot)$ to the exact sequence~(\ref{s2}).
To obtain the result, we must check that the composition
$Q_{2n,4}\ra A(2n)[2]\ra Q_{2n,4}[2]$ is nonzero. 
A fortiori, it is sufficient to check that the following composition 
of nontrivial morphisms
$\x_4\ra Q_{2n,4}\ra A(2n)[2]\ra Q_{2n,4}[2]\ra \x_{2n}[2]$
is nonzero.
The composition $\x_4\ra Q_{2n,4}\ra A(2n)[2]$ is nonzero by step 4 and 
the composition $A(2n)\ra Q_{2n,4}\ra \x_{2n}$ is nonzero by definition.
Finally, the composition $\x_4\ra A(2n)[2]\ra \x_{2n}[2]$ 
is nonzero by lemma~\ref{compos}.
                                 
6. $G_{2n}$ is an exceptional object. 

Apply the functor $\Hom(\cdot,G_{2n})$ to the exact sequence~(\ref{s2}). 
Using results of steps~1~and~5, we get $\Hom(G_{2n},G_{2n})=\C$,
$\Ext^i(G_{2n},G_{2n})=0$ for $i\ne 0$.

Lemma~\ref{Gexceptional} is proved. 
\end{proof}

\begin{proof}[Proof of theorem~\ref{EC1}]

It follows from proposition~\ref{Ext} that the collection
$(A, A(1),\ldots,A(2n-1),A(2n),\x_{2n},x_{2n-1},\ldots,\x_6,\x_5)$
is exceptional.
By lemma~\ref{Gexceptional}, objects $G_{2n}$ and $H_{2n+1}$ are exceptional.
It remains to check that the following groups vanish:

1. $\Ext^i(G_{2n},A(k))=0$ for $k=0,\ldots, 2n-1$.

This is obvious. Actually, by construction
$G_{2n}$ is  an object of the subcategory
\begin{equation}
\langle A(2n),\x_{2n},\x_{2n-2},\ldots, \x_6, \x_4\rangle;
\label{subcat}
\end{equation}
on the other hand, the modules $A,A(1),\ldots,A(2n-1)$  are right orthogonal
to this subcategory.

2. $\Ext^i(\x_j,G_{2n})=0$ for $j=5,6,\ldots,2n-1,2n$.

For even $j$ this follows from
$G_{2n}=L\langle A(2n),\x_{2n},\x_{2n-2}, \ldots,\x_6\rangle(\x_4)[n-2]$
and general properties of mutations, see~\cite{Bo}.
For odd $j$ we note that by proposition~\ref{Ext}
$\x_j$ is left orthogonal to the subcategory~(\ref{subcat}) and $G_{2n}$
belongs to this subcategory.

3. $\Ext^i(A(2n),G_{2n})=0$.

This is proved in step 1 of lemma~\ref{Gexceptional}.

4. $\Ext^i(H_{2n+1},A(k))=0$ for $k=1,\ldots, 2n$,
 $\Ext^i(\x_{j},H_{2n+1})=0$ for $j=6,\ldots, 2n$.

This follows from the exact sequence~(\ref{s4}) and proposition~\ref{Ext}.

5. $\Ext^i(H_{2n+1},A)=0$ and $\Ext^i(\x_5,H_{2n+1})=0$.

Apply the functors $\Hom(\cdot,A)$ and $\Hom(\x_5,\cdot)$
to the exact sequence~(\ref{s4}). To obtain the result, we should show that
the compositions of nontrivial maps 
$A(2n+1)\to\x_{2n+1}\to A[2]$ and $\x_5\to A(2n+1)[2]\to\x_5[2]$
are nontrivial. This is proved in lemma~\ref{compos}.

6. $\Ext^i(H_{2n+1},G_{2n})=0$.

Indeed, we have
$G_{2n}\in\langle A(2n),\x_{2n},\ldots, \x_6, \x_4\rangle\subset%
\langle A(2n+1),\x_{2n+1}\rangle^{\perp}\ni H_{2n+1}$.

This finishes the proof of theorem~\ref{EC1}.
\end{proof}

\begin{proof}[\bf Remark]
 It is not hard to check that for all $n\ge 2$ the collection
$$(A, A(1),A(2), \ldots, A(2n-2), G_{2n-1}, A(2n-1),
G_{2n}, A(2n),  \x_{2n}, \x_{2n-1},\ldots, \x_6, \x_5)$$
is exceptional. 
\end{proof}

\begin{theorem}\label{full}
The exceptional collections
\begin{gather}
(A, A(1), A(2),\ldots, A(2n-1), G_{2n}, A(2n), H_{2n+1}, \x_{2n}, \x_{2n-1},
\ldots,\x_6, \x_5) \quad\text{and} \label{ec_2}\\
(A, A(1),A(2), \ldots, A(2n-2), G_{2n-1}, A(2n-1),
G_{2n}, A(2n),  \x_{2n}, \x_{2n-1},\ldots, \x_6, \x_5)
\label{ec_1}
\end{gather}
in the category $\D^b(\qgr(A))$ are full.
\end{theorem}

\begin{proof}
First we prove that the collection~(\ref{ec_1}) is full.

Let $\DD_0$ be the subcategory of $\D^b(\qgr(A))$ generated
by objects of the collection~(\ref{ec_1}).
We claim that the modules $A(k)$ for all $k\in \Z$ belong to $\DD_0$.

As it was pointed before lemma~\ref{Gexceptional}, 
$G_{2n}=L_{\langle A(2n),\x_{2n},\x_{2n-2} 
\ldots \x_8,\x_6\rangle}(\x_4)[n-2]$, hence $\x_4\in \DD_0$. 
Similarly, $G_{2n-1}=L_{\langle A(2n-1),\x_{2n-1},\x_{2n-3} 
\ldots \x_7,\x_5\rangle}(\x_3)[n-2]$ and $\x_3\in\DD_0$. 

Consider the first exact sequence from lemma~\ref{kozhul}.
Twisting it by $k$, we obtain
\begin{equation}
0\ra A(k-3) \ra A(k-1)\oplus A(k-2) 
\ra A(k)\ra \left(A/{(x_0,x_1)}\right)(k)\ra 0.
\label{s5}
\end{equation}
For $k=3,\ldots,2n$ the terms
$A(k-3),A(k-1)\oplus A(k-2)$, and $A(k)$ of the complex~(\ref{s5}) 
belong to $\DD_0$, therefore we get $\left(A/{(x_0,x_1)}\right)(k)\in\DD_0$. 
Further, consider the exact sequence
$0\ra \x_{k-2n+1}\ra \left(A/{(x_0,x_1)}\right)(k)\ra \x_k\ra 0$
of type~(\ref{xqx}). 
For $k=3,\ldots,2n$ its terms
$\x_{k-2n+1}$ and $\left(A/{(x_0,x_1)}\right)(k)$ belong to $\DD_0$, 
it follows that $\x_{k-2n+1}\cong\x_{k+2n-2}\in\DD_0$.
Thus the objects
$\x_{2n+1}, \x_{2n+2},\ldots, \x_{4n-4},\x_{4n-3}=\x_0,\x_1$
belong to the subcategory~$\DD_0$. Let's prove that $\DD_0$ contains $\x_2$.

Consider the second exact sequence from lemma~\ref{kozhul}.
Twisting it by $2n$, we get
$0\ra A \ra A(2n-1)\oplus A(1)\ra A(2n)\ra Q_{2n,-2n+4}\ra 0$.
Since three left terms in this sequence belong to the subcategory $\DD_0$,
the right term $Q_{2n,-2n+4}$ also belongs to $\DD_0$.
We know that $Q_{2n,-2n+4}$ has a filtration~(\ref{filtr}) with quotients
isomorphic to the modules
$\x_{-2n+4}, \x_{-2n+6},\ldots, \x_{-2}, \x_0, \x_2,\ldots,\x_{2n}$.
All the above modules, except $\x_2$, belong to~$\DD_0$. 
Therefore $\x_2\in\DD_0$.

We see that the subcategory $\DD_0$ contains the objects $\x_j$ for all $j$.
It follows that $\DD_0$ contains the objects $\left(A/{(x_0,x_1)}\right)(k)$ for all $k$.                       
Now consider the complex~(\ref{s5}). We proved that for any integer $k$ 
the right term of~(\ref{s5}) belongs to $\DD_0$.
Taking $k=2n+1, 2n+2, \ldots $ and arguing as above we deduce that  
objects $A(2n{+}1), A(2n{+}2), \ldots$ belong to the subcategory $\DD_0$.
Similarly, putting $k=2, 1, 0, -1, \ldots $ we show that $\DD_0$
contains $A(-1), A(-2),\ldots$

We want to prove that the subcategory $\DD_0$ is equivalent to the 
whole category $\D^b(\qgr(A))$. Since a category generated by an  
exceptional collection is admissible (see~\cite{Bo}), it suffices to show
that the right orthogonal
$\DD_0^{\perp}=\{\,Y\in\D^b(\qgr(A))\mid\Hom(\DD_0,Y)=0\,\}$
to~$\DD_0$ is zero.

Assume the converse. Let $Y$ be an object in $\DD_0^{\perp}$ such that $Y\ne 0$.
Choose an integer $j_0$ such that $H^{j_0}(Y)\ne~0$.  
From~\cite[prop. 7.4(1), 3.11(3)]{AZ} it follows that for
any finitely generated graded $A$-module $M$ for $k\gg 0$ we have
$\Hom_{\qgr(A)}(A(-k),M)=M_k$ and $\Ext^i_{\qgr(A)}(A(-k),M)=0$ if $i>0$.
Choose an integer $k$ such that 
$\Ext^i(A(-k),H^j(Y))=0$ if $i>0, j\in\Z$ and 
$\Hom(A(-k),H^{j_0}(Y))\ne~0$. 
Use the spectral sequence with $E_2^{pq}=\Ext^q(A(-k),H^p(Y))$
to compute $\Hom^j(A(-k),Y)$.
By the choice of $k$,  $E_2^{pq}=0$ for $q>0$.
Thus  the spectral sequence degenerates in the second term,
$E_2^{pq}=E_{\infty}^{pq}$, and we get
$\Hom^{j_0}(A(-k),Y)=E_2^{j_00}=\Hom_{\qgr(A)}(A(-k),H^{j_0}(Y))\ne~0$.
This contradiction concludes the proof.

Now we prove that the collection~(\ref{ec_2}) is full. 
Let us show that the mutation
$G'=L_{\langle A(2n-1),G_{2n},A(2n)\rangle}(H_{2n+1})$ 
is isomorphic to $G_{2n-1}[k]$ for some $k$; this implies
that the collection~(\ref{ec_2}) is full.
Indeed, $G'$ belongs to the intersection of two orthogonals
$\SS={}^{\perp}\langle A,\ldots,A(2n-2)\rangle\cap%
\langle A(2n-1),G_{2n},\ldots,\x_5\rangle^{\perp}$.
Since the collection~(\ref{ec_1}) is full, the subcategory $\SS$
is generated by the exceptional object $G_{2n-1}$ and therefore it
is equivalent to the derived category of vector spaces $\D^b(\C\rm{-vect})$. 
Any exceptional object in $\D^b(\C\rm{-vect})$ is of the form $\C[k]$;
it follows that $G'\cong G_{2n-1}[k]$ for some $k$.

Theorem~\ref{full} is completely proved. 
\end{proof}

\section{Computation of the morphism algebra, associated with exceptional 
collection.}
\label{alg}

Starting with the collection
$(\x_{2n},\x_{2n-2},\ldots,\x_8,\x_6)$ we can obtain the collection
\linebreak
$(\x_6,R_{\x_6}(\x_8),R_{\langle\x_8,\x_6\rangle}(\x_{10}),\ldots,%
R_{\langle\x_{2n-2},\ldots,\x_6\rangle}(\x_{2n}))$ by a series of mutations.
The latter collection is equivalent  up to shifts of objects to the collection
$(Q_{6,6},Q_{8,6},\ldots,Q_{2n-2,6},Q_{2n,6})$,
which has a nice property --- higher $\Ext$ groups between 
its objects are zero.

Mutating as above the collection~(\ref{ec_2}), 
we obtain the collection
\begin{equation}
(A, A(1), \ldots, A(2n-1), 
G_{2n}, A(2n), H_{2n+1},
Q_{6,6},Q_{8,6}, \ldots, Q_{2n,6},
Q_{5,5}, Q_{7,5}, \ldots, Q_{2n-1,5}).
\label{ec_3}
\end{equation}
The number of nonzero higher $\Ext$ groups for this collection is much less
than for the collections~(\ref{ec_1})~and~(\ref{ec_2}).
In this section we explicitly describe the morphism algebra
for the collection~(\ref{ec_3}). That is, we compute
groups $\Hom$ and $\Ext$ between objects and describe
the composition law.

Let
$\wdt{A_k}$ be the subspace of $A_k$ defined as follows.
If $m\le n-2$, then $\wdt{A_{2m}}$ is the subspace of $A_{2m}$ 
generated by all 
monomials except $x_1^m$; in other cases $\wdt{A_k}$ is the whole
space $A_k$.
By definition, $\wdt{A_k}$ is the homogeneous component $(G_0)_k$  of degree $k$
of the ideal $G_0=\ker(A\ra Q_{0,4-2n})$.

\begin{predl}\label{extsincoll}
All nontrivial $\Ext$ groups between objects of the
collection~(\ref{ec_3}) are listed below.
\begin{enumerate}
\item[(a)] $\Hom(A(k),A(l))=A_{l-k}$, polinomials from $A_{l-k}$ act
on $A(k)$ by multiplication;

\item[(b)] $\Hom(A(k),Q_{j+2r,j})=\C=\langle x_1^{r+(j-k)/2}\rangle$ if
$j\le k\le j+2r, 
k \equiv j \pmod{2}$, the generator sends $1$ into $[x_1^{r+(j-k)/2}]$
(further on,  $[\quad ]$ denotes a coset in a quotient module);

\item[(c)] $\Hom(Q_{j+2r,j},Q_{j+2s,j})=\C{=}\langle x_1^{s-r}\rangle$ if
$j=5,6$, $r{\le}s$,
the generator sends $[1]$ into $[x_1^{s-r}]$;

\item[(d)] $\Hom(A(k),G_{2n})=\wdt{A_{2n-k}}$,\\
$\Hom(G_{2n},A(2n))=\C$, the generator is the embedding of the ideal, \\
$\Ext^1(G_{2n},A(2n))=\C$,
$\Ext^1(G_{2n},Q_{2n,6})=\C$;

\item[(e)] $\Hom(A(k),H_{2n+1})=A_{2n+1-k}$ if $0 \le k \le 2n$,
polinomials from $A_{2n+1-k}$ act by multiplication,\\
$\Hom(G_{2n},H_{2n+1})=\C=\langle x_0\rangle$, the generator
acts by multiplication by $x_0$,\\
$\Hom(H_{2n+1},Q_{2n-1,5})=\C$, the generator is given by 
$f(x)\mapsto \frac1{x_1}[f(x)]$.
\end{enumerate}
\end{predl}

\begin{proof}
(a) See proposition~\ref{Ext}.

(b) To check this, note that $Q_{j+2r,j}$ has
the filtration~(\ref{filtr}). 

(c) To compute $\Ext^i(Q_{j+2r,j},Q_{j+2s,j})$, 
apply the functor $\Hom(Q_{j+2r,j},\cdot)$ to the exact sequence
$0{\ra}Q_{j+2r,j}{\ra}Q_{j+2s,j}{\ra}Q_{j+2s,j+2r+2}{\ra}0$.
Notice that $\Ext^i(Q_{j+2r,j},Q_{j+2s,j+2r+2})=0$ to get the result.
Use the filtrations~(\ref{filtr}) to check the 
equalities $\Ext^i(Q_{6+2r,6},Q_{5+2s,5})=0$.

(d) To calculate $\Ext^i(A(k),G_{2n})$ apply the functor $\Hom(A(k),\cdot)$
to the exact sequence~(\ref{s2}).
The groups $\Ext^i(G_{2n},A(2n))$ can be computed
by applying the functor $\Hom(\cdot,A(2n))$ to the sequence~(\ref{s2});
the required groups $\Ext(Q_{2n,4},A(2n))$ were computed in step 4 
of lemma~\ref{Gexceptional}. 
It is not hard to check that
$\Ext^1(G_{2n},\x_{2n})=\C$, other groups
$\Ext(G_{2n},\x_j)$, where $5\le j\le 2n$, are zero. 
Now, passing from~$\x$ to~$Q$ we can сompute the groups 
$\Ext^i(G_{2n},Q_{j+2r,j})$. 

(e) The groups $\Ext^i(A(k),H_{2n+1})$ can be easily computed by the definition
of $H_{2n+1}$.
Note that there exists a unique morphism  $x_0\colon A(2n)\to A(2n+1)$ from
objects $A(2n),\x_{2n},\ldots,\x_6,\x_4$ to objects $A(2n+1),\x_{2n+1}$,
this yields the statements concerning $\Ext(G_{2n},H_{2n+1})$. 
To calculate $\Ext^i(H_{2n+1},Q)$ apply the functor 
$\Hom(\cdot,Q_{j+2r,j})$ to the sequence~(\ref{s4}). 
It remains to check that the formula
$f(x)\mapsto \frac1{x_1}[f(x)]$ gives a well-defined nonzero morphism
$H_{2n+1}\ra Q_{2n-1,5}$. 
\end{proof}

In the next proposition we give an explicit description of composition maps
$\Ext^i(E_l,E_m)\otimes \Ext^j(E_k,E_l)\ra \Ext^{i+j}(E_k,E_m)$
for objects $E_i$ of the collection~(\ref{ec_3}).
We omit trivial cases where one of the groups 
$\Ext^i(E_l,E_m)$, $\Ext^j(E_k,E_l)$ or $\Ext^{i+j}(E_k,E_m)$ is zero.

\begin{predl}
The composition law in the morphism algebra for the collection~(\ref{ec_3})
is the following: 
\begin{itemize}
\item[(a)] $A(k)\ra A(l)\ra A(m)$.  The composition $A_{m-l}\otimes A_{l-k}\ra
A_{m-k}$ is the multiplication of polinomials.
\item[(b)] $A(j+2k)\ra A(j+2l)\ra Q_{j+2r,j}$.  The composition 
$\C\langle x_1^{r-l}\rangle\otimes A_{2l-2k}\ra \C\langle x_1^{r-k}\rangle$ 
is given by $x_1^{r-l}\otimes f\mapsto c_{l-k}x_1^{r-l}$,
where $c_{l-k}$ is the coefficient of $x_1^{l-k}$ in $f\in A_{2l-2k}$.
\item[(c)] $A(j{+}2k)\ra Q_{j+2r,j}\ra Q_{j+2s,j}$.  The composition
$\C\langle x_1^{s-r}\rangle\otimes \C\langle x_1^{r-k}\rangle\ra%
\C\langle x_1^{s-k}\rangle$ 
 is the multiplication.
\item[(d)] $Q_{j+2q,j}\ra Q_{j+2r,j}\ra Q_{j+2s,j}$.  The composition
$\C\langle x_1^{s-r}\rangle\otimes \C\langle x_1^{r-q}\rangle\ra \C\langle x_1^{s-q}\rangle$ 
is the multiplication.
\item[(e)] $A(k)\ra A(l)\ra G_{2n}$.  The composition
$\wdt{A_{2n-l}}\otimes A_{l-k}\ra \wdt{A_{2n-k}}$ is a restriction of
the multiplication $A_{2n-l}\otimes A_{l-k}\ra A_{2n-k}$. 
\item[(f)] $A(k)\ra G_{2n}\ra A(2n)$.  The composition
$\C\otimes \wdt{A_{2n-k}}\ra A_{2n-k}$ is the  multiplication.
\item[(g)] $G_{2n}\ra A(2n)[1]\ra Q_{2n,6}[1]$. \\
 The composition
$\Hom(A(2n),Q_{2n,6})\otimes \Ext^1(G_{2n},A(2n))\ra \Ext^1(G_{2n},Q_{2n,6})$
is nonzero.
\item[(h)] $A(k)\ra G_{2n}\ra H_{2n+1}$.  The composition
$\C\langle x_0\rangle\otimes \wdt{A_{2n-k}}\ra A_{2n+1-k}$ 
is the multiplication.
\item[(i)] $G_{2n}\ra A(2n)\ra H_{2n+1}$.  The composition
$\C\langle x_0\rangle\otimes \C\ra \C\langle x_0\rangle$ is the multiplication.
\item[(j)] $A(2k+1)\ra H_{2n+1}\ra Q_{2n-1,5}$ ($k\ge 2$).  The composition 
$\C\otimes A_{2n-2k}\ra \C\langle x_1^{n-k-1}\rangle$  is given by
$1\otimes f\mapsto c_{n-k} x_1^{n-k-1}$, where $c_{n-k}$ is the coefficient
of $x_1^{n-k}$ in $f$.
\end{itemize}
\end{predl}

\begin{proof}
These statements, except for probably part (g), 
easily follow from the description of
$\Hom$ groups given in  proposition~\ref{extsincoll}.
Let us check that the composition of nonzero 
morphisms 
$G_{2n}\ra A(2n)[1]\ra Q_{2n,6}[1]$ is nonzero.
In fact, we show that  the composition
$\x_4[-1]\ra Q_{2n,4}[-1]\ra G_{2n}\ra A(2n)[1]\ra Q_{2n,6}[1]\ra \x_{2n}[1]$
is nontrivial, where the morphism $Q_{2n,4}[-1]\ra G_{2n}$ is given by 
the extension~(\ref{s2})
and the left and right morphisms are definied in subsection~\ref{defA}.
Indeed, by the proof of proposition~\ref{extsincoll}d this composition
is a composition of nontrivial morphisms
$\x_4[-1]\ra A(2n)[1]\ra \x_{2n}[1]$.
It is nonzero by lemma~\ref{compos}. 
\end{proof}

\end{document}